\documentclass[10pt]{article}
\usepackage{xcolor}
\usepackage{algpseudocode}
\usepackage{algorithm}
\usepackage{amssymb}
\usepackage{multicol}
\usepackage{lipsum}
\usepackage{colonequals, mathtools}
\usepackage{graphicx}
\usepackage[mathscr]{euscript}
\usepackage{amsthm}
\usepackage{verbatim}
\usepackage{algorithmicx}
\usepackage{algorithm}
\usepackage{algpseudocode}
\usepackage{fixltx2e}
\usepackage{hyperref}
\usepackage{amsmath, epsfig}
\usepackage{calc}
\usepackage{tikz}
\usepackage{color}
\usepackage{caption}
\allowdisplaybreaks
\usetikzlibrary{decorations.markings}
\tikzstyle{vertex}=[circle, draw, inner sep=0pt, minimum size=6pt]

\usepackage{geometry}
\newcommand{\vertex}{\node[vertex]}

\newtheorem{theorem}{Theorem}
\newtheorem{lemma}{Lemma}
\newtheorem{proposition}{Proposition}  
\newtheorem{corollary}{Corollary}
\newtheorem{definition}{Definition}

\newtheorem{example}{Example}

\newcommand{\vol}{\mathrm{vol}}

\newcommand{\p}{\mathbb{P}}

\title{Graphs with many strong orientations}
\author{
Sinan Aksoy \thanks{Research supported in part by NSF grant DMS-1427526, ``The Rocky Mountain - Great Plains Graduate Research Workshop in Combinatorics".} \thanks{Department of Mathematics, University of California, San Diego, {\tt saksoy@ucsd.edu}.}  \and Paul Horn \footnotemark[1] \thanks{Department of Mathematics, University of Denver, {\tt paul.horn@du.edu}.}
  }

\begin{document}

\maketitle

\begin{abstract}
We establish mild conditions under which a possibly irregular, sparse graph $G$ has ``many" strong orientations. Given a graph $G$ on $n$ vertices, orient each edge in either direction with probability $1/2$ independently. We show that if $G$ satisfies a minimum degree condition of $(1+c_1)\log_2{n}$ and has Cheeger constant at least $c_2\frac{\log_2\log_2{n}}{\log_2{n}}$, then the resulting randomly oriented directed graph is strongly connected with high probability. This Cheeger constant bound can be replaced by an analogous spectral condition via the Cheeger inequality. Additionally, we provide an explicit construction to show our minimum degree condition is tight while the Cheeger constant bound is tight up to a $\log_2\log_2{n}$ factor. \\

\noindent \textbf{Key words.} strong orientations, strongly connected, Cheeger constant \\

\noindent \textbf{AMS subject classifications.} 05C20, 05C40, 05C50
\end{abstract}

\section{Introduction}

Given an undirected graph, the classic Robbins' theorem \cite{robbins} provides a simple criterion to determine whether there exists an orientation of its edges yielding a strongly connected digraph. Namely, graphs admitting a {\it strong orientation} are $2$-edge connected graphs. This theorem has since been extended to multigraphs \cite{tindell}; moreover, there are linear time algorithms which detect strong orientations and construct them whenever possible \cite{tarjan}. While the existence and construction of strong orientations are well-understood topics, the task of {\em counting} strong orientations is less straightforward.

Counting strong orientations has natural applications. One example comes from statistical mechanics, where the Eulerian orientations (i.e. strong orientations for which each vertex has equal in and out-degree) of a $4$-regular graph are the allowable configurations in two-dimensional {\em ice-type} models used to study crystals with hydrogen bonds \cite{ice}. More generally, the problem of counting the number of Eulerian orientations and strong orientations of a given graph $G$ are special cases of evaluating its Tutte polynomial, $T(G; x,y)$. Specifically, the number of Eulerian orientations of a given graph $G$ is $T(G;0,-2)$ and number of strong orientations is $T(G;0,2)$ \cite{vergnas}.  

Counting Eulerian and strong orientations has been shown to be $\#P$-hard (see \cite{winkler} and \cite{vertigan} respectively). Instead of exact counting, researchers have approximated the number of strong orientations for particular classes of graphs. In the case of $\alpha$-dense graphs $G$ (i.e. graphs with minimum degree $\delta(G)>\alpha n$ for $0<\alpha<1$), Alon, Frieze, and Welsh \cite{fpras} developed a fully polynomial randomized approximation scheme for counting strong orientations. 

In this paper, we investigate a different type of question related to counting strong orientations: for which possibly sparse and irregular graphs are ``most" orientations strongly connected? In particular, we show that if a general graph $G$ satisfies a mild isoperimetric condition and mild minimum degree requirement, then a random orientation will be strongly connected with high probability. Our main result is as follows:

\begin{theorem} \label{mainThm}
Given any $\alpha > 0$ and $\xi > 4$, there exists an integer $N_0 = N_0(\alpha, \xi)$ such that for $n \geq N_0$, if $G$ is an $n$-vertex graph with minimum degree $\delta(G) \geq (1+\alpha)\log_2{n}$ and Cheeger constant $\displaystyle \Phi(G)>\xi \cdot \frac{\log_2\log_2 n}{\log_2 n}$, then a random orientation of $G$ is strongly connected with probability 
at least $\displaystyle 1-\frac{1+4\alpha \log_2{n}}{\alpha n^\alpha \log_2{n}}.$   
\end{theorem} 

Thus, a graph $G$ satisfying the conditions of Theorem $\ref{mainThm}$ has $(1-o(1))2^{e(G)}$ many strong orientations, where $e(G)$ denotes the number of edges of $G$. While we will provide a formal definition in Section $\ref{tools}$, we note that the Cheeger constant of a graph measures the fewest number of edges leaving a vertex set relative to the ``size" of that set. Beyond the bound on the Cheeger constant and the minimum degree requirement, we do not assume the graph necessarily satisfies additional structural properties; in particular, the graph is {\em not} assumed to be regular. Not assuming regularity increases the utility of the result, but introduces additional subtleties in the proof, particularly with regard to enumerating connected $k$ sets of the graph.

As we will show in Section $\ref{discussion}$, the minimum degree requirement is tight while the bound on the Cheeger constant is tight up to a $\log_2\log_2{n}$ factor. Since the normalized Laplacian eigenvalues of a general graph can be more efficiently computed than its Cheeger constant, it may be useful to reformulate the second condition in Theorem $\ref{mainThm}$ as a spectral condition via the Cheeger inequality.

\begin{corollary} \label{mainCor}
In Theorem $\ref{mainThm}$, the condition $\Phi(G)>\xi \cdot \frac{\log_2\log_2 n}{\log_2 n}$  may be replaced with
\[
\frac{\lambda_1(G)}{2}		 > \xi \cdot \frac{\log_2\log_2 n}{\log_2 n},
\]
\noindent where $\lambda_1{(G)}$ denotes the second eigenvalue of the normalized Laplacian of $G$.
\end{corollary}

The paper is organized as follows: In Section $\ref{tools}$, we establish our main tools and introduce relevant notation. In Section $\ref{discussion}$ we briefly discuss the isoperimetric condition and minimum degree requirement in Theorem $\ref{mainThm}$. In Section $\ref{proof}$, we present the proof of Theorem $\ref{mainThm}$.

\section{Main Tools and Notation} \label{tools}

For a vertex subset $X \subseteq V(G)$, $\vol(X):= \sum_{v\in X}\deg(v)$ and $e(X,\overline{X})$ denotes the number of edges between $X$ and its complement $\overline{X}$. The {\em Cheeger ratio} of $X$ is
\[
\Phi(X) = \frac{e(X,\overline{X})}{\min(\vol(X), \vol(\overline{X}))},
\] 

\noindent and the {\em Cheeger constant} of the graph $G$ is $\Phi(G) = \min_{X \subseteq V(G)} \Phi(X)$. When the graph $G$ in question is clear from context, we will simply denote its Cheeger constant $\Phi(G)$ as $\Phi$ and its minimum degree $\delta(G)$ as $\delta$. \\

The {\em normalized Laplacian} of a graph $G$ is
\[
\mathscr{L}=I-D^{-1/2}AD^{-1/2},
\]

\noindent where $A$ is the adjacency matrix and $D$ is the diagonal degree matrix. The eigenvalues of $\mathscr{L}$ are labeled in increasing order:
\[
0=\lambda_0\leq\lambda_1\leq\dots\leq\lambda_{n-1}\leq2.
\]

Here, we emphasize that the spectral condition in Corollary $\ref{mainCor}$ only pertains to the second eigenvalue and thus makes no additional assumptions about the {\em spectral gap}, $\sigma=\max_{i\geq1}|1-\lambda_i|$, which is the key parameter in controlling the {\em discrepancy} of a graph. Thus, while we assume a bound on $|1-\lambda_1|$, we do not assume an additional bound on the other end of the spectrum, $|1-\lambda_{n-1}|$, beyond the trivial bound that holds for any graph, $\lambda_{n-1}\leq2$.

The well-known Cheeger inequality describes the relationship between the normalized Laplacian eigenvalues of a graph and its Cheeger constant:

\begin{theorem}[Cheeger Inequality \cite{fan}]
Let $\Phi$ be the Cheeger constant of a graph $G$ and $\lambda_1$ the second eigenvalue of the normalized Laplacian. Then: 

\[
2\Phi \geq \lambda_1 \geq \frac{\Phi^2 }{2}.
\]
\end{theorem}

A consequence of this, by which Corollary $\ref{mainCor}$  follows immediately from Theorem $\ref{mainThm}$, is that for any set $X \subseteq V(G)$ with $\vol(X) \leq \frac{1}{2}\vol(G)$, 
\begin{equation}
e(X,\overline{X}) \geq \frac{\lambda_1}{2} \vol(X).  \label{eq:cheeger}
\end{equation}
As an aside, this uses only the bound $2\Phi \geq \lambda_1$, so we are not using the full strength of Cheeger's inequality.  Indeed, on graphs the lower bound $\lambda_1 \leq 2\Phi$ is easily proven (for instance, see Lemma 2.1 in \cite{fan}).  In the Riemannian manifold case,  Cheeger's inequality only refers to the lower bound on $\lambda_1$ in terms of $\Phi$ -- the upper bound on $\lambda_1$ in terms of $\Phi$ is Buser's inequality \cite{buser}. Nonetheless, we stick with the convention in graph theory and refer to $(\ref{eq:cheeger})$ as following from Cheeger's inequality.  

Finally, we will utilize standard asymptotic notation: we say a function $f(n)=O(g(n))$ if for all sufficiently large values of $n$ there exists a positive constant $c$ such that $|f(n)|\leq c\cdot |g(n)|$; similarly, we write $f(n)=\Omega(g(n))$ if $g(n)=O(f(n))$, and $f(n)=\Theta(g(n))$ if both  $f(n)=O(g(n))$ and $f(n)=\Omega(g(n))$. Lastly, $f(n)=o(g(n))$ if $\lim_{n\rightarrow \infty}\frac{f(n)}{g(n)}=0$.
 
\section{Discussion of Theorem $\ref{mainThm}$} \label{discussion}

Before we proceed with the proof of Theorem $\ref{mainThm}$, we briefly discuss the minimum degree requirement and Cheeger constant bound. First, we show that each of these conditions, taken on their own, do not ensure that a random orientation of a graph yields a strongly connected directed graph with {\em any} nonzero limiting probability. For instance, Figure~\ref{fig:bridgeMinDegree} illustrates the so-called {\em barbell graph} on $n$ vertices, which has minimum degree a factor of $n$ but possesses a bridge. Similarly, the graph obtained by connecting a single vertex to exactly one vertex of $K_{n-1}$ has Cheeger constant always at least $1/2$ but again contains a bridge. Thus, neither condition in Theorem $\ref{mainThm}$ is sufficient in ensuring the result. 

Next, we show our minimum degree requirement is sharp while the bound on the Cheeger constant is sharp up to a $\log_2\log_2{n}$ factor.  In order to do this, we will make use of the fact that if $G$ is a random $d$-regular graph, for $d = c \log_2 n$, then $G$ has a Cheeger constant bounded away from zero.  Such results were known for fixed $d$ dating to the work of Bollob\'as \cite{bollobas}.  

For non-constant degree, as in our case, the easiest approach to such a result is to appeal to the spectra.  The study of spectra of random regular graphs has a long history, culminating most famously in Friedman's proof of Alon's second eigenvalue conjecture \cite{fried}: random regular graphs of fixed degree $d$ have second eigenvalue of the adjacency matrix $2\sqrt{d-1} + \epsilon$ for any $\epsilon > 0$, with high probability.  This, again unfortunately for our work, focuses on the case with constant degree.  Fortunately for our purposes, Broder, Frieze, Suen and Upfal \cite{BFSU} showed that the technique used by Kahn and Szemeredi in \cite{KS} works in the case that $d = o(n^{1/2})$, and shows that the second eigenvalue of the adjacency matrix is $O(\sqrt{d})$ for such graphs.  In terms of normalized Laplacian eigenvalues, this shows that $\lambda_{1} \geq 1 - O(d^{-1/2})$ in this regime, and through Cheeger's inequality random $d$-regular graphs have Cheeger constant satisfying $\Phi > \frac{1}{4}$ with high probability so long as $n$ is sufficiently large.  We mention that this problem is still attracting attention, as just recently, Cook, Goldstein and Johnson \cite{CGJ} proved that the second adjacency eigenvalue for a random $d$-regular graph is still $O(\sqrt{d})$ for $d=o(n^{2/3})$.

We now use the fact that a $\log_2 n$ regular graph has Cheeger constant at least $1/4$ with high probability when considering the following example, which shows our minimum degree requirement is sharp. 

\begin{figure}[t]
\[\begin{tikzpicture}
	\vertex (c1) at (45:1){};
	\vertex (c2) at (90:1){};
	\vertex (c3) at (135:1){};
	\vertex (c4) at (180:1){};
	\vertex(c5) at (225:1){};
	\vertex(c6) at (270:1){};
	\vertex(c7) at (315:1){};
	\vertex(c8) at (360:1){};
	\vertex (d1) at (3.71,.71){};
	\vertex (d2) at (3,1){};
	\vertex (d3) at (2.29,.71){};
	\vertex (d4) at (2,0){};
	\vertex(d5) at (2.29,-.71){};
	\vertex(d6) at (3,-1){};
	\vertex(d7) at (3.71,-.71){};
	\vertex(d8) at (4,0){};
	\path 
		(c1) edge (c2)
		(c1) edge (c3)
		(c1) edge (c4)
		(c1) edge (c5)
		(c1) edge (c6)
		(c1) edge (c7)
		(c1) edge (c8)
		(c2) edge (c1)
		(c2) edge (c3)
		(c2) edge (c4)
		(c2) edge (c5)
		(c2) edge (c6)
		(c2) edge (c7)
		(c2) edge (c8)
		(c3) edge (c2)
		(c3) edge (c1)
		(c3) edge (c4)
		(c3) edge (c5)
		(c3) edge (c6)
		(c3) edge (c7)
		(c3) edge (c8)
		(c4) edge (c2)
		(c4) edge (c3)
		(c4) edge (c1)
		(c4) edge (c5)
		(c4) edge (c6)
		(c4) edge (c7)
		(c4) edge (c8)
		(c5) edge (c2)
		(c5) edge (c3)
		(c5) edge (c4)
		(c5) edge (c1)
		(c5) edge (c6)
		(c5) edge (c7)
		(c5) edge (c8)
		(c6) edge (c2)
		(c6) edge (c3)
		(c6) edge (c4)
		(c6) edge (c5)
		(c6) edge (c1)
		(c6) edge (c7)
		(c6) edge (c8)
		(c7) edge (c2)
		(c7) edge (c3)
		(c7) edge (c4)
		(c7) edge (c5)
		(c7) edge (c6)
		(c7) edge (c1)
		(c7) edge (c8)
		(c8) edge (c2)
		(c8) edge (c3)
		(c8) edge (c4)
		(c8) edge (c5)
		(c8) edge (c6)
		(c8) edge (c1)
		(c8) edge (c7)
		(d1) edge (d2)
		(d1) edge (d3)
		(d1) edge (d4)
		(d1) edge (d5)
		(d1) edge (d6)
		(d1) edge (d7)
		(d1) edge (d8)
		(d2) edge (d1)
		(d2) edge (d3)
		(d2) edge (d4)
		(d2) edge (d5)
		(d2) edge (d6)
		(d2) edge (d7)
		(d2) edge (d8)
		(d3) edge (d2)
		(d3) edge (d1)
		(d3) edge (d4)
		(d3) edge (d5)
		(d3) edge (d6)
		(d3) edge (d7)
		(d3) edge (d8)
		(d4) edge (d2)
		(d4) edge (d3)
		(d4) edge (d1)
		(d4) edge (d5)
		(d4) edge (d6)
		(d4) edge (d7)
		(d4) edge (d8)
		(d5) edge (d2)
		(d5) edge (d3)
		(d5) edge (d4)
		(d5) edge (d1)
		(d5) edge (d6)
		(d5) edge (d7)
		(d5) edge (d8)
	         (d6) edge (d2)
		(d6) edge (d3)
		(d6) edge (d4)
		(d6) edge (d5)
		(d6) edge (d1)
		(d6) edge (d7)
		(d6) edge (d8)
		(d7) edge (d2)
		(d7) edge (d3)
		(d7) edge (d4)
		(d7) edge (d5)
		(d7) edge (d6)
		(d7) edge (d1)
		(d7) edge (d8)
		(d8) edge (d2)
		(d8) edge (d3)
		(d8) edge (d4)
		(d8) edge (d5)
		(d8) edge (d6)
		(d8) edge (d1)
		(d8) edge (d7)
		(d4) edge (c8)

	;
\end{tikzpicture}\]
\caption{Two copies of $K_{n/2}$ connected by an edge.}
\label{fig:bridgeMinDegree}
\end{figure}
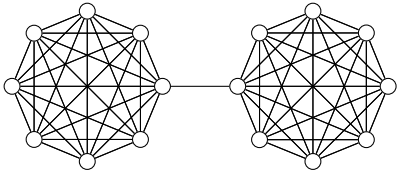

\begin{example} \label{ex: g'}
Let $G'$ be a random $t$ regular graph on $N=2^{t}$ vertices.
\end{example}

\begin{proposition} \label{prop: g'}
$G'$ has minimum degree $\log_2{N}$ and, with high probability, Cheeger constant at least $\frac{1}{4}$. However, a random orientation of $G'$ is disconnected with limiting probability at least $1-\frac{1}{e}$.
\end{proposition}

\begin{proof}
We show a random orientation of $G'$ is disconnected with limiting probability at least $1-\frac{1}{e}$. Since $G'$ is $\log_2{N}$ regular, the probability a vertex is a sink in a random orientation is $\frac{1}{N}$. Assume the vertices are labeled and let $B_i$ denote the event that vertex $i$ is a sink. For fixed $k$, define

\begin{align*}
S^{(k)}=\sum_{\{i_1,\dots,i_k \} \in {V(G') \choose k}} \p(B_{i_1} \cap \dots \cap B_{i_k} ).\nonumber
\end{align*}

By Brun's sieve \cite[Theorem~8.3.1]{alonspencer}, if we show that for every fixed $k$

\[
\lim_{N \to \infty} S^{(k)}= \frac{1}{k!},
\]
then the limiting probability there are no sinks in a random orientation of $G'$ is $\frac{1}{e}$. Note that if $\{i,j\} \in E(G')$, then $\p(B_i \cap B_j)=0$. Thus, we may rewrite the sum for $S^{(k)}$ as over all independent sets with $k$ vertices.  Accordingly, since we need each of the $t=\log_2{N}$ edges for each of the $k$ vertices to be oriented so that each is a sink, $\p(B_{i_1} \cap \dots \cap B_{i_k} )=\left(\frac{1}{2^t} \right)^k=\frac{1}{N^k}$. At most, every $k$-subset of $V(G')$ is an independent set, yielding the upper bound

\begin{align*}
S^{(k)}&\leq {N \choose k} \frac{1}{N^k} \sim  \frac{1}{k!},
\end{align*}
and at least, there are $\frac{1}{k!}\cdot N(N-\log_2{N})\dots(N-(k-1)\log_2{N})\geq \frac{(N-k \log_2{N})^k}{k! \cdot N^k}$ independent sets of size $k$, yielding the lower bound

\begin{align*}
S^{(k)}&\geq \frac{(1-\frac{k\log_2{N}}{N})^k\cdot N^k}{k! \cdot N^k}\sim \frac{1}{k!}.
\end{align*}
\end{proof}

Having shown that the minimum degree condition in Theorem $\ref{mainThm}$ is sharp, we now use $G'$ to construct a graph $G$ to show the Cheeger constant condition in Theorem $\ref{mainThm}$ is sharp up to a $\log_2\log_2{n}$ factor.

\begin{example}\label{ex: g}

For any integer $c > 1$, consider the graph $G$ on $n$ vertices obtained from $G'$ by appending to each vertex in $G'$ vertex disjoint  complete graphs on $ct$ vertices.  (Equivalently, $G$ is constructed by taking $N$ vertex disjoint cliques on $ct$ vertices, selecting from each of them a distinguished vertex, and amongst the $N$ distinguished vertices, placing a $t$ regular random graph). See Figure $\ref{fig:CheegerExample}$ for one example of this construction.
\end{example}

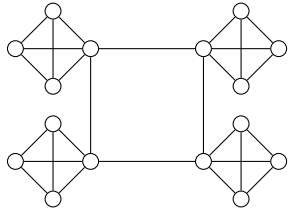
\begin{figure}[t]
\[\begin{tikzpicture}
	\vertex (c1) at (0,0) {};
	\vertex (c2) at (1.5,0) {};
	\vertex (c3) at (0,1.5) {};
	\vertex (c4) at (1.5,1.5) {};
	\vertex (c31) at (-1,1.5) {};
	\vertex (c32) at (-0.5,2) {};
	\vertex(c33) at (-0.5, 1.0) {};
	\vertex (c41) at (2.5,1.5) {};
	\vertex (c42) at (2,1) {};
	\vertex(c43) at (2, 2) {};
	\vertex (c11) at (-1,0) {};
	\vertex (c12) at (-0.5,0.5) {};
	\vertex(c13) at (-0.5, -0.5) {};
	\vertex (c21) at (2.5,0) {};
	\vertex (c22) at (2,-0.5) {};
	\vertex(c23) at (2, 0.5) {};
	\path 
		(c1) edge (c2)
		(c2) edge (c4)
		(c4) edge (c3)
		(c1) edge (c3)
		(c3) edge (c31)
		(c3) edge (c32)
		(c3) edge (c33)
		(c31) edge (c32)
		(c31) edge (c33)
		(c32) edge (c33)
		(c4) edge (c41)
		(c4) edge (c42)
		(c4) edge (c43)
		(c41) edge (c42)
		(c41) edge (c43)
		(c42) edge (c43)
		(c1) edge (c11)
		(c1) edge (c12)
		(c1) edge (c13)
		(c11) edge (c12)
		(c11) edge (c13)
		(c12) edge (c13)
		(c2) edge (c21)
		(c2) edge (c22)
		(c2) edge (c23)
		(c21) edge (c22)
		(c21) edge (c23)
		(c22) edge (c23)
	;
\end{tikzpicture}\]
\caption{The graph $G$ in Example $\ref{ex: g}$ with $t=c=2$.}
\label{fig:CheegerExample}
\end{figure}
  
\begin{proposition} \label{prop: g}
$G$ has minimum degree $\Omega(\log_2{n})$ and Cheeger constant $\Phi(G)=\Omega(\log_2^{-1}{n})$. However, a random orientation of $G$ is disconnected with limiting probability at least $1-\frac{1}{e}$.  
\end{proposition}

\begin{proof}
First, recalling that $G$ is constructed by appending {\em disjoint} complete graphs to each vertex in $G'$, Proposition $\ref{prop: g'}$ immediately implies a random orientation of $G$ is disconnected with limiting probability at least $1-\frac{1}{e}$. Next, we examine examine the minimum degree and Cheeger constant of $G$. Note that the graph $G$ is on $n=ctN$ vertices, and $\log_2 n = t + \log_2(ct)$.  For $t$ large enough, the minimum degree in the graph (which is $ct-1$) is at least $\frac{c \log_2 n}{2}$ and the maximum degree is $(c+1)t-1 < 2c\log_2n$.  For any subset $X \subseteq V(G)$ with $\vol(X) < \vol(G)/2$, we will show that 
\[
\frac{e(X,\bar{X})}{\vol(X)} = \Omega(\log_2^{-1}{n}).  
\]
Note that since every vertex has degree $\Theta(\log_2{n})$ it suffices to show that for all subsets of cardinality at most $\frac{n}{2}$, 
\[
\frac{e(X,\bar{X})}{|X|} = \Omega(1).  
\]
This is what we shall do. Let $S \subseteq V(G)$ denote the vertices of $G'$ (contained as a subgraph of $G$.)  Let $S_1, \dots, S_N$ denote the vertices contained (respectively) in each of the $N$ cliques. Note $|S_i \cap S| = 1$ for all $i$, as there is a unique distinguished vertex per clique.  Fix a set $X \subseteq V(G)$.  Define the sets
\begin{align*}
S' &= X \cap S \\
T_1 &= \{x \in X: x \in S_i \mbox{ with } (S_i \cap S') \neq \emptyset \mbox{ for some $i \in [N]$}\}\\
T_2 &= \{x \in X: x \in S_i \mbox{ with } (S_i \cap S') = \emptyset  \mbox{ for some $i \in [N]$}\}.
\end{align*}
Note that $S' \subseteq T_1$ and $T_1$ and $T_2$ partition $X$.   
We observe that 
\[
\frac{e(X,\bar{X})}{|X|}=\frac{e(T_1,\bar{X})+ e(T_2, \bar{X})}{|T_1| + |T_2|}. 
\]
By the real number inequality
\[
\frac{a+b}{c+d} \geq \min\left\{\frac{a}{c}, \frac{b}{d} \right\},
\]
valid for positive $a,b,c,d$, it suffices to show that both $\frac{e(T_1,\bar{X})}{|T_1|}$ and $\frac{e(T_2,\bar{X})}{|T_2|}$ are both $\Omega(1)$ (unless one of them is $\frac{0}{0}$ -- note that both of them cannot be since $X$ is non-empty).

We begin by proving that $\frac{e(T_2,\bar{X})}{|T_2|} = \Omega(1)$ so long as $T_2$ is non-empty.  Let $r_i = |S_i \cap T_2|$.  Note that $r_i \leq ct-1$ for every $i$, as the distinguished vertices are {\it not} in $T_2$.  Further note that since $S_i$ is a clique, the $r_i$ vertices in $S_i$ are adjacent to all remaining $ct-r_i$ vertices in the clique which are in $\bar{X}$.  Thus
\[
e(T_2,\bar{X}) = \sum_i r_i(ct - r_i) \geq \sum_i r_i = |T_2|,
\]
so $e(T_2,\bar{X})/|T_2| = \Omega(1)$. \\

It is slightly more complicated to bound $\frac{e(T_1,\bar{X})}{|T_1|}$.  Similarly, we let $n_i = |S_i \cap T_1|$.  Let $m = |T_1 \cap S|$.  Then
\[
e(T_1,\bar{X}) = e(T_1 \cap S, \bar{X} \cap S) + \sum_i n_i(ct - n_i). 
\]
Since $G'$ has $\Phi(G') \geq \frac{1}{4}$, 

$$
e(T_1 \cap S, \bar{X} \cap S) \geq \frac{1}{4} \min\{m, N-m\} \cdot \log_2 N.    
$$
If $m \leq \frac{9N}{10}$ this is sufficient to show $\frac{e(T_1,\bar{X})}{|T_1|} = \Omega(1)$, since $|T_1|=O(m \log_2{N})$ and $e(T_1 \cap S, \bar{X} \cap S) = \Omega(m \log_2{N})$. Otherwise, if $m>\frac{9N}{10}$, without loss of generality $n_1, n_2, \dots, n_m$ are positive.  Consider the function:
\[
f(n_1,n_2,n_3, \dots, n_m) = \sum_i n_i(ct-n_i).  
\]
Note that if $x \geq y$,
\[
(x+1)(ct-(x+1)) + (y-1)(ct-(y-1)) - \left(x(ct-x) + y(ct-y)\right) = 2(y-x-1) < 0.
\]
Thus, for any two arguments of the function $f$, increasing the larger by 1 while decreasing the smaller by 1 decreases the function. Since $f$ is symmetric in its variables, we may relabel them so that $n_1\geq \dots \geq n_m$ and repeatedly apply the above observation to yield:

\[f(n_1,n_2,\dots, n_m) \geq f(ct,ct,\cdots, ct,*,1,1,\dots,1,1),\]
so that the arguments sum to $\sum n_i$ and $1 \leq * \leq ct$.  Since $\sum n_i \leq \frac{n}{2}$ and $m > \frac{9N}{10}$, this means that there are at least $\frac{4N}{10}$ 1's, so
\[
f(n_1,n_2,n_3, \dots, n_m) \geq  f(ct,ct,\cdots, ct,*,1,1,\dots,1,1) \geq \frac{4N}{10} \cdot 1(ct-1) = \Omega(n).
\]
This shows $e(T_1, \bar{X}) = \Omega(n)$, hence $\frac{e(T_1,\bar{X})}{|T_1|} = \Omega(1)$.  Thus we have shown $\Phi(G)=\Omega(\log_2^{-1}{n})$.
 \end{proof}

\section{Proof of Theorem $\ref{mainThm}$} \label{proof}

Our general approach to proving Theorem $\ref{mainThm}$ is based on the observation that a directed graph is strongly connected if and only if every nonempty proper subset $X\subsetneq V(G)$ has an edge both entering and leaving it. Namely, we bound the probability that every connected set $X \subseteq V(G)$ with $\vol(X) \leq \vol(G)/2$  has an edge both entering and leaving it. 

\begin{definition}
For a subset $X$ of vertices let $\mathcal{B}_X$ be the event that $\vol(X) \leq \vol(G)/2$, $X$ is connected in $G$ and $X$ has either no edges oriented into it or out of it.  Note only the third property here is random -- if $X$ does not have one of the first two properties, $\p(\mathcal{B}_X) = 0$ deterministically.  We further define
\[
\mathcal{B}_k = \bigcup_{\substack{X \subseteq V(G)\\|X|=k}} \mathcal{B}_X.
\]
\end{definition}

We estimate $\p(\bigcup_k \mathcal{B}_k) \leq \sum_{k} \p(\mathcal{B}_k)$ by dividing $k$ into two regimes. First we prove that every small subset (where $|X| \leq c\log_2 n$) has an edge entering and leaving: \\

\noindent{\bf Regime 1:} We claim $\sum_{k=1}^{\alpha/2 \log_2 n} \p(\mathcal{B}_k) < \frac{4}{n^{\alpha}}$.  \\

\begin{proof}
We begin by noting that for a given set $X$ of size $k$, there are at most
${k \choose 2}$ edges induced on $X$ and hence, recalling that $\delta$ denotes the minimum degree, there are at least $\delta k - {k \choose 2}$ edges leaving. Note that in this regime, $\delta k - {k \choose 2}>0$ since $k \leq \frac{\alpha}{2}\log_2 n$.  For a given set $X$,  
\[
\p(\mathcal{B}_X) \leq 2^{-\delta  k + {k \choose 2}+1 },
\]
and this gives an estimate
\[
\p(\mathcal{B}_k) \leq {n \choose k} 2^{-\delta k + {k \choose 2} + 1} \equalscolon  b_k.
\]
We note that if $k \leq \frac{\alpha}{2}\log_2 n$,
\begin{align*}
\frac{b_{k+1}}{b_k} &= \frac{{n \choose k+1} 2^{-\delta \cdot (k+1) + {k+1 \choose 2} + 1}}{{n \choose k} 2^{-\delta k + {k \choose 2} + 1}}\\
&= \frac{(n-k)2^{k}}{(k+1)2^{\delta }} \\ &\leq \frac{2^{k}}{n^{\alpha}} \leq \frac{1}{2}.
\end{align*}
  Then 
\[
\sum_{k=1}^{\frac{\alpha}{2}\log_2 n} \p(\mathcal{B}_k) \leq 2\p(\mathcal{B}_1) \leq 2n 2^{1-(1+\alpha) \log_2{n}} = 4n^{-\alpha}.   
\]
\end{proof}

\noindent{\bf Regime 2:}  We claim $\sum_{k \geq \alpha/2 \log_2n}^n \p(\mathcal{B}_k)  \leq \frac{1}{\alpha n^\alpha \log_2{n}}$.

\begin{proof}

For large sets, we must take greater care -- the number of edges that {\it could} be induced in sets is much larger, we utilize our lower bound on the Cheeger constant to ensure many edges leave each set.  Since the number of potential $k$ sets grows large as well, we will restrict attention to counting only connected sets so as to not over count. 

To this end, we will enumerate connected $k$ sets by considering rooted spanning trees in $G$, which we will consider labeled.  The shape of spanning trees, can of course, vary wildly.  For the purposes of this work we will enumerate them by their {\it exposure sequence}.

\begin{definition}
An exposure sequence $\pi=(\pi_1,\pi_2,\dots,\pi_{k-1})$ of a labeled rooted spanning tree on $k$ vertices is determined as follows: newly label the vertices in breadth-first order, with ties broken by the original labeling of the tree. Then, under this new labeling $\pi_i$ is the number of children of vertex $v_i$ in the tree. See Figure \ref{fig:expoEx} for an example.

\end{definition}

Therefore, an exposure sequence of a rooted spanning tree on $k$ vertices is an (ordered) list of $(k-1)$ non-negative integers $(\pi_1, \pi_2, \dots, \pi_{k-1})$ satisfying $\sum_{i \leq j} \pi_i \geq j$ and $\sum_{i=1}^{k-1} \pi_i = k-1$.  A given exposure sequence of $k-1$ numbers uniquely determines the shape of the rooted, spanning tree on $k$ vertices. Since these vertices are labeled in breadth-first order, the $k$th vertex is necessarily a leaf of the tree, so by convention we have $\pi_k=0$. We note that an exposure sequence for a rooted spanning tree on $k$ vertices can be thought of as a staircase walk on the square lattice from $(0,0)$ to $(k-1, k-1)$ which never crosses the diagonal. Namely, the staircase walk corresponding to exposure sequence $\pi$ is formed by taking $\pi_i$ steps east and 1 step north for $i=1,\dots,k-1$ (see Figure $\ref{fig:expoEx}$). Thus, counting all possible exposure sequences is equivalent to counting all Dyck paths on the square lattice, which is given by the Catalan numbers: $c_{k-1}=\frac{1}{k} {2(k-1) \choose k-1}$.

\begin{figure}
\centering
\begin{tikzpicture}[scale=0.6]
    \tikzstyle{every node}=[circle,draw]
    \node[ultra thick] {1}
        child { 
        node {2} 
        		child {
        		node {4}
		child{ node{6}}
		child{ node{5}}
		child {
		node{7}
			child{node{8}}
			}
		}
	child {node {3}}
        }
    ;
\end{tikzpicture}
\hspace{3cm}
\begin{tikzpicture} [scale=0.59]
\draw[step=1cm,gray,very thin] (0,0) grid (7,7);
\draw[thin, dashed] (0,0) -- (7,7);
\draw[ultra thick, -] (0,0) -- (1,0);
\draw[ultra thick, -] (1,0) -- (1,1);
\draw[ultra thick, -] (1,1) -- (3,1);
\draw[ultra thick, -] (3,1) -- (3,3);
\draw[ultra thick, -] (3,3) -- (6,3);
\draw[ultra thick, -] (6,3) -- (6,6);
\draw[ultra thick, -] (6,6) -- (7,6);
\draw[ultra thick, -] (7,6) -- (7,7);
\end{tikzpicture}
\caption{Left: a breadth-first vertex labeling of a rooted tree yielding exposure sequence $\pi=(1,2,0,3,0,0,1)$. Right: the staircase walk corresponding to exposure sequence $\pi=(1,2,0,3,0,0,1)$. }
\label{fig:expoEx}
\end{figure}
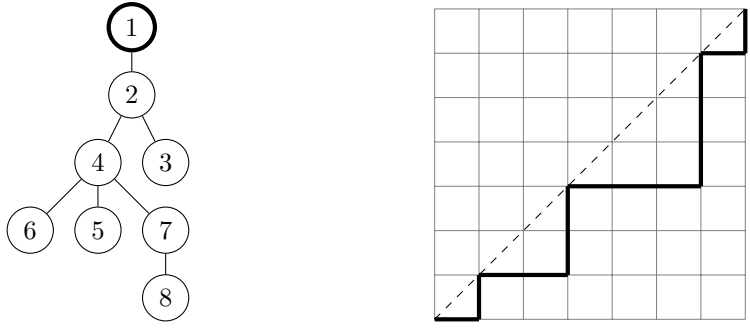

We will enumerate all of the rooted subtrees in $G$ on $k$ vertices by their exposure sequence. Our task now is to bound the following sum: 

\begin{multline}
\p(\mathcal{B}_k) \leq \sum_{\pi = (\pi_1, \pi_2 \dots, \pi_{k-1})} \sum_{v_1 \in V(G)} \sum_{\substack{\{v_2,v_3,\dots,v_{1+\pi_1}\} \\ \in {N(v_1) \choose \pi_1}}}\\* \sum_{\substack{\{v_{2+\pi_1}, v_{3+ \pi_1}, \cdots, v_{\pi_1+\pi_2+1}\} \\ \in {N(v_2) \choose \pi_2}}}
\cdots \sum_{\substack{\{v_{k-\pi_{k-1}+1}, \dots, v_{k}\}\\ \in {N(v_{k-1}) \choose \pi_{k-1}}}} \p(\mathcal{B}_X), \label{eqn:telescope}
\end{multline}    
where $X = \{v_1,\dots, v_k\}$ and ${N(v_i) \choose \pi_i}$ denotes the set of all sets of $\pi_i$ vertices adjacent to $v_i$ in the original graph $G$.  For any $X$ which is connected in the original graph,
\[
\p(\mathcal{B}_X) = \begin{cases} 0 & \mbox{if }\vol(X) > \frac{\vol(G)}{2} \\ 
2^{1-e(X,\bar{X})} & \mbox{if }\vol(X) \leq \frac{\vol(G)}{2}
\end{cases}.
\]
In the second case, 
\[
\p(\mathcal{B}_X) \leq 2^{1-e(X,\bar{X})} \leq 2^{1-\Phi \vol(X)} = 2^{1 - \Phi \sum_{v_i \in X}  \deg(v_i)}.
\]
Since this bounds $\p(\mathcal{B}_X)$  above by a positive quantity (and $\p(\mathcal{B}_X)$ is otherwise zero), the inequality
\begin{equation}
\p(\mathcal{B}_X) \leq 2^{1-\Phi \sum_{v_i \in X} \deg(v_i)}  \label{eqn:ineq}
\end{equation}
holds for every $X$.  We now use (\ref{eqn:ineq}) to bound the right hand side of (\ref{eqn:telescope}).  We wish to collapse a term of the form
\[
\sum_{{N(v_j) \choose \pi_j}} 2^{1-\Phi \deg(v_j)},
\]
as (after having already bounded each of the summands for $v_i$ where $i > j$) we will have ensured that the summand is independent of the $\pi_j$ vertices chosen.  Thus,
\[
\sum_{{N(v_j) \choose \pi_j}} 2^{1-\Phi \deg(v_j)} = {\deg(v_j) \choose \pi_j} 2^{1-\Phi \deg(v_j)}.
\]
We will give an upper bound of this term which is independent of $v_j$, depending only on $\pi_j$ and $\delta $, and this will allow us to continue collapsing the sum (\ref{eqn:telescope}). We find three different upper bounds for this term for the cases when $\pi_j=0$, $\pi_j=1$, or $\pi_j>1$. \\ 

\noindent {\bf Case 1}: $\pi_j=0$.

If $v_j$ is a leaf of the embedded spanning tree (which corresponds to $\pi_j = 0$), we
simply bound 
\[
2^{1-\Phi \deg(v_j)} \leq 2^{1-\Phi \cdot (1+\alpha) \log_2{n}}.  
\]

\noindent {\bf Case 2}: $\pi_j>1$.

Since the terms we are interested in have the general form:  ${\deg(v_j) \choose \pi_j} 2^{1-\Phi \deg(v_j)}$, we investigate the associated sequence defined by fixing $\pi_j$ and varying $\deg(v_j)$. In general, let 
\[
\kappa_{s,t} = {s \choose t} 2^{1-\Phi s},
\]
so that the terms appearing above are $\kappa_{\deg(v_j),\pi_j}$.  
Then for a fixed $t$ and varying $s$, the sequence $\kappa_{s,t}$ is unimodal.  We have that
\[
\frac{\kappa_{s,t}}{\kappa_{s+1,t}} = \left(1 - \frac{t}{s+1}\right) 2^{\Phi}.
\]
Thus the maximum of $\kappa_{s,t}$, for a fixed $t$, is achieved by the smallest $s$ such that $\kappa_{s+1,t} < \kappa_{s,t}$, yielding
\[
\frac{t}{s+1} <  1 - 2^{-\Phi},
\]
or equivalently 
\[
(s+1) > \frac{t}{1-2^{-\Phi}}. 
\]
Thus the maximum of $\kappa_{s,t}$ occurs when 
\[
s_{\max(t)}  = \left \lfloor \frac{t}{1 - 2^{-\Phi}} \right\rfloor.
\]
Indeed, extending the binomial coefficients to the reals in the usual way, the floor function can be dropped. Recalling that $\Phi,s\geq0$, for fixed $t$ we have:
\begin{align}
\kappa_{s,t} \leq {(1-2^{-\Phi})^{-1} t \choose t}2^{1-\Phi \cdot (1-2^{-\Phi})^{-1}t}.  \label{eq:binomreal}
\end{align}
Next, we use the entropy bound:

\[
{n \choose k} \leq \frac{n^n}{k^k(n-k)^{n-k}} = 2^{nH(k/n)},
\]
where $H(q)=-q\log_2{q}-(1-q)\log_2(1-q)$ is the binary entropy function. \\

Applying this bound to the binomial coefficient in (\ref{eq:binomreal}) yields:

\begin{align}
\log_2{{(1-2^{-\Phi})^{-1} t \choose t}} &\leq {(1-2^{-\Phi})^{-1} tH(1-2^{-\Phi})} \nonumber\\
& = {-(1-2^{-\Phi})^{-1} t \left[ (1-2^{-\Phi}) \log_2(1 - 2^{-\Phi}) + 2^{-\Phi} \log_2 (2^{-\Phi}) \right]} \nonumber \\
& = {t \left[ -\log_2(1 - 2^{-\Phi}) + \frac{\Phi}{2^{\Phi}-1} \right]}. \nonumber
\end{align}

Combining this upper bound with $(\ref{eq:binomreal})$ and simplifying, we have that

\begin{align}
\kappa_{s,t} &\leq 2^{1+t (-\log_2(1 - 2^{-\Phi})-\Phi)}. \label{eq:entbd} 
\end{align}

We will now provide constant upper bounds on the terms involving $\Phi$ in the exponent of (\ref{eq:entbd}). Setting $f(x)=-\log_2(1-2^{-x})$, we have $f'(x) = \frac{-1}{2^x-1}$ and so for $x>0$, 
\[
f(x)=f(1)+\int_1^x f'(t) \, dt = 1 + \int_x^1 \frac{1}{2^t -1} \, dt. 
\]
Since $1 + x \ln 2 \leq e^{x \ln 2}=2^x$, we have that for $x>0$, 
\[ 
f(x) \leq 1 + \int_x^1 \frac{1}{t \ln2} \, dt = 1-\log_2(x). 
\]

We may use this to bound (\ref{eq:entbd}), yielding
\begin{align*}
\kappa_{s,t} &<2^{1-t\log_2(\Phi) + t}.
\end{align*}
Although we will only apply this when $\pi_j>1$, this gives the general bound, good for any $\deg(v_j), \pi_j$ that 
\[
{\deg(v_j) \choose \pi_j}2^{1-\Phi \deg(v_j)} \leq 2^{1-\pi_j \log_2(\Phi) + \pi_j}.
\]

\noindent {\bf Case 3}: $\pi_j=1$.

In this case, the previous bound does not suffice for our purposes.  Here, we improve the bound by observing that our conditions imply that $\deg(v_j) > s_{max}(t).$   Indeed, our  condition that $\Phi > \xi \frac{\log_2\log_2 n}{\log_2 n}$ implies that for  $n$ sufficiently large, $(1+\alpha)\log_2{n} > (1-2^{-\Phi})^{-1}$.    Hence we are interested in $\kappa_{\deg(v_j),\pi_j}$ and by the unimodality of the $\kappa_{s,t}$ for $t$ fixed, we can derive the bound:

\begin{align}
{\deg(v_j) \choose \pi_j} 2^{1-\Phi \deg(v_j)} &< { (1+\alpha) \log_2 n \choose \pi_j} 2^{1-\Phi \cdot (1+\alpha) \log_2n} \nonumber \\
&< ((1+\alpha)\log_2 n)^{\pi_j} 2^{1-\Phi \cdot (1+\alpha)\log_2n} \nonumber \\
&= 2^{1+\pi_j \left[\log_2(1+\alpha) + \log_2\log_2 n\right] - \Phi \cdot (1+\alpha)\log_2{n}}, \label{eq:smallpibd}  
\end{align}
which for $\pi_j = 1$ simplifies to
\[
{\deg(v_j) \choose 1} 2^{1-\Phi\deg(v_j)} < 2^{1+ [\log_2(1+\alpha)+\log_2\log_2 n] - \Phi \cdot (1+\alpha) \log_2 n}.
\]

Collecting our results from Cases 1, 2, and 3, we have established the following:

\begin{align}
\sum_{{N(v_j) \choose \pi_j}} 2^{1-\Phi \deg(v_j)} &= {\deg(v_j) \choose \pi_j} 2^{1-\Phi \deg(v_j)} \nonumber \\&\leq \begin{cases} 
2^{1-\Phi \cdot (1+\alpha)\log_2 n} & \mbox{if $\pi_j = 0$}\\
2^{1+\left[\log_2(1+\alpha) + \log_2\log_2 n\right] - \Phi \cdot (1+\alpha)\log_2{n}} & \mbox{if $\pi_j = 1$} \\
2^{1+\pi_j \log_2(1/\Phi) + \pi_j} & \mbox{if $\pi_j > 1$}
\end{cases}. \label{genbd}
\end{align}

Before we collapse the sum (\ref{eqn:telescope}) using (\ref{genbd}), we make a few simple combinatorial observations concerning exposure sequences of rooted spanning trees. Recalling that a degree of a vertex in the spanning tree is $\pi_1$ for $v_1$, and $\pi_i + 1$ for $v_i$, we define the following: 

\begin{definition}
For an exposure sequence $\pi = (\pi_1, \dots, \pi_{k-1})$, let $$\ell(\pi) = 1 + |\{j \leq k-1: \pi_j=0\}|$$ denote the number of leaves of the spanning tree described by the sequence and we let 
\[ 
p(\pi) = |\{j \leq k-1: \pi_j = 1\}|.
\]
\end{definition}

\begin{lemma} \label{combobs}
For any exposure sequence $\pi$, we have
\begin{align}
\bullet& \ p(\pi)+\ell(\pi)\geq\frac{k}{2}, \nonumber \\
\bullet& \sum_{j: \pi_j\geq 2}\pi_j<k-p(\pi). \nonumber
\end{align}
\end{lemma}

\begin{proof}
For the first observation, note that if $p(\pi)+\ell(\pi)<\frac{k}{2}$, then there are at least $\frac{k}{2}$ terms in $\pi$ that are at least 2, yielding the contradiction:
\[
k=(k/2)\cdot 2\leq\sum_{i=1}^{k-1}\pi_i=k-1.
\]
And, for the second observation:
\begin{align}
k-1=\sum_{i=1}^{k-1}\pi_i&=\sum_{j: \pi_j\geq 2}\pi_j+\sum_{j: \pi_j= 1}\pi_j+\sum_{j: \pi_j=0}\pi_j\nonumber \\
&=\sum_{j: \pi_j\geq 2}\pi_j+p(\pi). \nonumber 
\end{align} 
\end{proof}

We now proceed to bound  $\p(\mathcal{B}_k)$ (\ref{eqn:telescope}). We will take logarithm here for readability so that every term would not appear in the exponent -- this should be viewed most naturally by exponentiating both sides. Iteratively applying (\ref{genbd}), we obtain that 
for a fixed $\pi = (\pi_1, \cdots, \pi_{k-1})$ and $v_1 \in V(G)$ 
\begin{align}
\log_2\p(\mathcal{B}_k)& \leq \log_2\left(\sum_{\substack{\{v_2,v_3,\dots,v_{1+\pi_1}\} \\ \in {N(v_1) \choose \pi_1}}}  \sum_{\substack{\{v_{2+\pi_1}, v_{3+ \pi_1}, \cdots, v_{\pi_1+\pi_2+1}\} \\ \in {N(v_2) \choose \pi_2}}}
\cdots \sum_{\substack{\{v_{k-\pi_{k-1}+1}, \dots, v_{k}\}\\ \in {N(v_{k-1}) \choose \pi_{k-1}}}} \p(\mathcal{B}_X)\right) \nonumber \\
&\leq \left[ \sum_{j: \pi_j \geq 2} 1+\pi_j \log_2(1/\Phi) + \pi_j \right]  + p(\pi)[1+\log_2(1+\alpha)+ \log_2\log_2n] \nonumber \\ &\hspace{1in}- (p(\pi) + \ell(\pi)) \Phi {\cdot} (1+\alpha)\log_2{n}+\ell(\pi)  \nonumber \\
&= \left[ \sum_{j: \pi_j \geq 2} \pi_j \log_2(1/\Phi) + \pi_j \right]  + p(\pi)[\log_2(1+\alpha)+ \log_2\log_2n] \nonumber \\ &\hspace{1in}- (p(\pi) + \ell(\pi)) \Phi {\cdot} (1+\alpha)\log_2{n}+k.  \nonumber 
\end{align}

Continuing, we apply Lemma \ref{combobs} to yield
\begin{align}
& \left[ \sum_{j: \pi_j \geq 2} \pi_j \log_2(1/\Phi) + \pi_j \right]  + p(\pi)[\log_2(1+\alpha)+ \log_2\log_2n] \nonumber \\ &\hspace{1in}- (p(\pi) + \ell(\pi)) \Phi {\cdot} (1+\alpha)\log_2{n}+k  \nonumber \\
& \leq (k - p(\pi))\left( \log_2(1/\Phi) + 1\right)+ p(\pi)[\log_2(1+\alpha)+ \log_2\log_2n] \nonumber \\ &
\hspace{1in} - \frac{k}{2} {\cdot} \Phi {\cdot} (1+\alpha) \log_2 n +k.\nonumber 
\end{align}

Next, using the fact that $\Phi > \xi \frac{\log_2 \log_2 n}{\log_2 n}$ for some (large) constant $\xi$, we obtain
\begin{align}
& (k - p(\pi))\left( \log_2(1/\Phi) + 1\right)+ p(\pi)[\log_2(1+\alpha)+ \log_2\log_2n] \nonumber \\ &
\hspace{1in} - \frac{k}{2} {\cdot} \Phi {\cdot} (1+\alpha) \log_2 n +k\nonumber \\
&< (k-p(\pi))\log_2{\log_2{n}}+(k-p(\pi))+p(\pi)\log_2{(1+\alpha)}\nonumber \\ & \hspace{1in} +p(\pi)\log_2{\log_2{n}} -\xi\frac{k}{2}(1+\alpha)\log_2{\log_2{n}}+k \nonumber \\
&\leq k\left(\log_2{\log_2{n}}\left[1-\frac{\xi}{2}(1+\alpha)\right]+\left(2+\log_2(1+\alpha)\right)\right).\label{eqn:currbd}
\end{align}

Finally, for $\xi>4$ and $n$ sufficiently large we have:
\begin{align}
 & k\left(\log_2{\log_2{n}}\left[1-\frac{\xi}{2}(1+\alpha)\right]+\left(2+\log_2(1+\alpha)\right)\right) \nonumber \\
 &\leq  k\left(\log_2{\log_2{n}}\left[1-2(1+\alpha)\right]\right) \nonumber \\
 &= -k {\cdot} \frac{\alpha\log_2{\log_2{n}}}{2}\left(2\alpha^{-1}+4\right) \nonumber \\
 &\leq -k(2\alpha^{-1}+4). \nonumber
\end{align}

Therefore, by assuming $n$ and $\xi$ are large enough, for any fixed $\pi$ and $v_1 \in V(G)$ we have that:

\begin{align*}
&\sum_{\substack{\{v_2,v_3,\dots,v_{1+\pi_1}\} \\ \in {N(v_1) \choose \pi_1}}}  \sum_{\substack{\{v_{2+\pi_1}, v_{3+ \pi_1}, \cdots, v_{\pi_1+\pi_2+1}\} \\ \in {N(v_2) \choose \pi_2}}}
\cdots \sum_{\substack{\{v_{k-\pi_{k-1}+1}, \dots, v_{k}\}\\ \in {N(v_{k-1}) \choose \pi_{k-1}}}} \p(\mathcal{B}_X) \\
&\hspace{1.5in} \leq 2^{-(2\alpha^{-1} + 4)k }.
\end{align*}

Using the above bound and recalling that there are $c_{k-1}=k^{-1}{2(k-1) \choose k-1}$ many exposure sequences, we now bound all of (\ref{eqn:telescope}) as:

 \begin{align*}
 \p(\mathcal{B}_k) &\leq \sum_{\pi = (\pi_1, \pi_2 \dots, \pi_{k-1})} \sum_{v_1 \in V(G)} \sum_{\substack{\{v_2,v_3,\dots,v_{1+\pi_1}\} \\ \in {N(v_1) \choose \pi_1}}} \\ &\hspace{1in}\sum_{\substack{\{v_{2+\pi_1}, v_{3+ \pi_1}, \cdots, v_{\pi_1+\pi_2+1}\} \\ \in {N(v_2) \choose \pi_2}}}
 \cdots \sum_{\substack{\{v_{k-\pi_{k-1}+1}, \dots, v_{k}\}\\ \in {N(v_{k-1}) \choose \pi_{k-1}}}} \p(\mathcal{B}_X) \\
 &\leq    \sum_{\pi = (\pi_1, \pi_2 \dots, \pi_{k-1})} \sum_{v_1 \in V(G)} 2^{-(2\alpha^{-1}+4)k} \\
 &\leq nk^{-1}{2(k-1) \choose k-1} 2^{-(2\alpha^{-1}+4)k}\\
 &\leq nk^{-1}4^{k-1}2^{-(2\alpha^{-1}+4)k} \\
 &= 2^{\log_2{n}-\left(\log_2{k}+(2\alpha^{-1}+2)k+2\right) }\\
 &\leq 2^{-\left(\log_2{k}+2k+2\right) },
\end{align*}   
where, in the last inequality, we used that  $k \geq \frac{\alpha}{2} \log_2 n$.  Thus:
\begin{align*}
\sum_{k=\frac{\alpha}{2}\log_2 n}^{n} \p(\mathcal{B}_k) &\leq \sum_{k=\frac{\alpha}{2}\log_2 n}^{n} 2^{-\left(\log_2{k}+2k+2\right) } \\ &\leq 2\cdot 2^{-\left(\log_2({\frac{\alpha}{2}\log_2{n}})+\alpha\log_2{n}+2\right) }\\ &= \frac{1}{\alpha n^\alpha\log_2{n}}.
\end{align*}

This completes our estimate for Regime 2. 

\end{proof}
Finally, combining the estimates we derived in each regime, we see that
\[
\p\left(\bigcup_{k=1}^n \mathcal{B}_k\right) \leq \sum_{k=1}^{n} \p(\mathcal{B}_k) \leq \frac{1+4\alpha \log_2{n}}{\alpha n^\alpha \log_2{n}},
\]
and thus with probability at least $1-\frac{1+4\alpha \log_2{n}}{\alpha n^\alpha \log_2{n}}=1-o(1)$, a random orientation of $G$ is strongly connected, completing our proof of Theorem $\ref{mainThm}$.  \\
\qed \\

\noindent {\bf Acknowledgment}: The authors would like to thank Fan Chung for originating this problem and for helpful comments that improved subsequent drafts of the paper. We are also grateful for constructive input from anonymous referees, including their suggestion of a cleaner bound that simplified the exposition. This research was supported in part by NSF grant DMS-1427526, ``The Rocky Mountain - Great Plains Graduate Research Workshop in Combinatorics".


\begin{thebibliography}{1}


\bibitem{fpras} N. Alon, A.M. Frieze, and D.J.A. Welsh, Polynomial time randomised approximation schemes for Tutte-Grothendieck invariants: the dense case, {\em Random Structures Algorithms}, {\bf 6} (1995), pp. 459-478.

\bibitem{alonspencer} N.~Alon and J.~H.~Spencer, {\it The Probabilistic Method}, 3rd ed., Wiley--Interscience, New York, 2008.

 \bibitem{tindell} F. Boesch and R. Tindell, Robbins' theorem for mixed multigraphs, {\em Amer. Math. Monthly}, {\bf 87} (1980) pp. 716-719.

\bibitem{bollobas} B. Bollob\'{a}s, The isoperimetric number of random regular graphs, {\em European J. Combin.}, {\bf 9} (1988) pp. 241-244.

\bibitem{BFSU}  A. Z. Broder, A. M. Frieze, S. Suen, and E. Upfal, Optimal construction of edge disjoint paths in random graphs, {\em SIAM J. Comput.}, {\bf 28} (1999) pp. 541Ð573.

\bibitem{buser} P. Buser, A note on the isoperimetric constant, {\em Ann. Sci. \'Ecole Norm. Sup.}, {\bf 15} (1982) pp. 213-230. 
 
 
 \bibitem{fan}F. Chung, {\em Spectral Graph Theory}, AMS Publications, 1997.
 
 
  
 \bibitem{tarjan} F. Chung, M. Garey, and R. Tarjan, Strongly connected orientations of mixed multigraphs, {\em Networks} {\bf 15} (1985), pp. 477-484.
 
 \bibitem{CGJ} N. Cook, L. Goldstein, T. Johnson, Size biased couplings and the spectral gap for random regular graphs, arxiv:1510.06013 [math.PR] (2015).
 
 \bibitem{fried} J. Friedman, A proof of Alon's second eigenvalue conjecture and related problems, {\em Mem. Amer. Math. Soc.}, {\bf 195} (2008), viii+100.
 
 \bibitem{KS} J. Friedman, J. Kahn, and E. Szemeredi, On the second eigenvalue of random regular graphs, {\em Proceedings of the 21st Annual ACM Symposium on Theory of Computing} (1989), pp. 587-598.
 
 \bibitem{vergnas} M. Las Vergnas, Convexity in oriented matroids, {\em J. Combin. Theory Ser. B}, {\bf 29} (1980), pp. 231-243.
 
 \bibitem{ice} E.H. Lieb, Residual entropy of square ice, {\em Phys. Rev.} {\bf 162}, (1967) pp. 162-172.
 
 \bibitem{winkler} M. Mihail and P. Winkler, On the number of Eulerian orientations of a graph, Technical Memorandum TM-ARH-018829, Bellcore (1991).
  
  \bibitem{robbins} H.E. Robbins, A theorem on graphs, with an application to a problem of traffic control, {\em Amer. Math. Monthly}, {\bf 46} (1939), pp. 281-283.
  
\bibitem{vertigan} D.L. Vertigan and D.J.A. Welsh, The computational complexity of the Tutte plane: the bipartite case, {\em Combin, Probab. Comput.}, {\bf 1} (1992), pp. 181-187.
   

  




  \end{thebibliography}
\end{document}